\documentclass[journal,twoside,web]{ieeecolor}
\usepackage[dvipsnames]{xcolor}
\usepackage{generic}
\usepackage{cite}
\usepackage{amsmath,amssymb,amsfonts}
\usepackage{graphicx}
\usepackage{textcomp}
\usepackage{subcaption}
\usepackage{lcsys}

\usepackage{amsthm}

\usepackage{tikz-cd}
\usetikzlibrary{arrows.meta}
\tikzcdset{
  arrow style=tikz,
  diagrams={>={Stealth[length=4pt,width=3.8pt]}}
}

\newtheorem{problemEnv}{Problem}

\newtheorem{remarkEnv}{Remark}
\newenvironment{remark}[1][]{\begin{remarkEnv}}{\hfill$\blacklozenge$\end{remarkEnv}}

\newtheorem{theoremEnv}{Theorem}

\newtheorem{lemmaEnv}{Lemma}

\newtheorem{corollaryEnv}{Corollary}

\newtheorem{propositionEnv}{Proposition}
\newenvironment{proposition}[1][]{\begin{propositionEnv}}{\hfill$\vartriangle$\end{propositionEnv}}

\newtheorem{definitionEnv}{Definition}
\newenvironment{definition}[1][]{\begin{definitionEnv}}{\hfill$\bullet$\end{definitionEnv}}

\newcommand{\dd}{\mathrm{d}}

\newcommand{\eqdef}{\triangleq}
\newcommand{\R}{\mathbb{R}}
\newcommand{\X}{\mathcal{X}}
\newcommand{\K}{\mathcal{K}}
\newcommand{\SO}[1]{\mathsf{SO}(#1)}
\newcommand{\so}[1]{\mathfrak{so}(#1)}
\DeclareMathOperator*{\argmin}{argmin}
\DeclareMathOperator{\spec}{spec}
\DeclareMathOperator{\Log}{Log}
\DeclareMathOperator{\Exp}{Exp}

\makeatletter
\let\NAT@parse\undefined
\makeatother
\usepackage[colorlinks=true,allcolors=blue]{hyperref}

\begin{document}

\def\BibTeX{{\rm B\kern-.05em{\sc i\kern-.025em b}\kern-.08em
    T\kern-.1667em\lower.7ex\hbox{E}\kern-.125emX}}
\markboth{\journalname, VOL. XX, NO. XX, XXXX 2017}
{Author \MakeLowercase{\textit{et al.}}: Preparation of Papers for IEEE Control Systems Letters (August 2022)}

\title{Learning Neural Maximal Lyapunov Functions on $\mathsf{SO}(n)$}

\author{Adeel Akhtar$^{a}$, 
Matthieu Barreau$^{b}$ 
\thanks{This work is partially supported by the Wallenberg AI, Autonomous Systems and Software Program (WASP) funded by the Knut and Alice Wallenberg Foundation and Digital Futures, and NJIT's startup funds.}
\thanks{$^{a}$ Department of Mechanical and Industrial Engineering,  New Jersey Institute of Technology, NJ, USA {\tt\small adeel.akhtar@njit.edu}.}
\thanks{$^{b}$ Department of Decision and Control Systems, Digital Futures, KTH Royal Institute of Technology, Stockholm, Sweden {\tt\small barreau@kth.se}. Both authors contributed equally.}%
}

\maketitle
\thispagestyle{empty}
 
\begin{abstract}
Establishing stability guarantees for dynamical systems on Lie groups is a fundamental challenge, as classical Lyapunov methods developed for Euclidean spaces do not directly transfer to curved geometries. In this paper, we propose a framework for learning maximal Lyapunov functions for systems evolving on the special orthogonal group $\mathsf{SO}(n)$. Theoretically, we introduce a neural Lyapunov architecture based on the logarithmic map with proven approximation capabilities, and we formulate the learning problem via a Zubov-type characterization of the maximal region of attraction. A key technical contribution is the derivation of explicit, numerically tractable formulas for the derivative of the logarithmic map, enabling training through a two-phase algorithm that balances computational efficiency and accuracy. Empirically, we validate the approach on a low-dimensional nonlinear system.
\end{abstract}

\begin{IEEEkeywords}
Lie groups, PINNs, Neural Lyapunov Functions, Learning.
\end{IEEEkeywords}

\section{Introduction}
\label{sec:introduction}
\IEEEPARstart{M}{any} engineering, biological, and physical systems evolve under structural geometric constraints that restrict their state to nonlinear manifolds rather than Euclidean spaces. Examples arise in rigid-body mechanics~\cite{MarRat1999}, in synchronization problems~\cite{ChaBanMah2025}, and in grand unified theories \cite{baez2010GUT}, to cite a few, where the system configuration is represented by a rotation matrix evolving on the special orthogonal group $\SO{n}$. Such systems are naturally modeled as dynamical systems on Lie groups, where the state in $\SO{n}$ evolves according to differential equations that preserve the orthogonality constraint. For these applications, establishing stability guarantees is crucial and is usually assessed using Lyapunov functions \cite{khalil2002nonlinear}.

Despite their central role in control theory, constructing Lyapunov functions for a generic nonlinear dynamical system remains a challenge. For linear time-invariant systems, global stability is equivalent to the existence of a quadratic Lyapunov function, computable via linear matrix
inequalities~\cite{boyd1994linear}. For nonlinear systems, an accurate Lyapunov function is generally not quadratic~\cite{khalil2002nonlinear}, and it
becomes even more important when estimating the largest region of attraction (ROA) around an equilibrium point. Zubov~\cite{zubov1961methods}, along with Vannelli and Vidyasagar~\cite{vannelli1985maximal} characterized the maximal ROA as the sublevel set of a \emph{maximal Lyapunov function} satisfying a certain partial differential equation (PDE), but lacked the computational means to approximate it.

More recently, polynomial parameterizations combined with sum-of-squares techniques~\cite{jones2021converse,henrion2013convex} and rational Lyapunov constructions~\cite{valmorbida2017region} have been proposed to solve this PDE, but their applicability is limited to low-dimensional systems with a specific structure. 
Residual-based neural methods, inspired by physics-informed learning~\cite{karniadakis2021physics}, have recently been used to approximate Lyapunov functions by enforcing Lyapunov or Zubov-type conditions at sampled collocation points~\cite{gaby2022lyapunov,liu2023physics,barreau2024supervised}. 
These approaches avoid fixed-grid discretizations of the state space, which makes them attractive for systems whose intrinsic dimension is too large for classical gridding-based methods.
The unsupervised setting was treated in~\cite{barreau2024supervised} with local guarantees.

Despite this progress, no existing work addresses the learning of a maximal Lyapunov function on $\SO{n}$. 
The difficulties in Euclidean settings are amplified on manifolds: classical constructions are not directly applicable, the geometry of the state space must be respected, and although $\SO{n}$ is embedded in $\R^{n \times n}$, it is a nonlinear manifold of intrinsic dimension $m_n=n(n-1)/2$ that can be locally parametrized by $m_n$ coordinates. Leveraging this intrinsic low-dimensional structure while preserving the group geometry is the central challenge we address in this work.

This paper makes two contributions. First, we introduce a neural Lyapunov architecture on $\SO{n}$, yielding positive-definite functions with universal approximation capabilities on the rotation group.  
Second, we derive explicit, numerically tractable expressions for the derivative of the logarithmic map and propose a two-phase training algorithm.

\noindent\textbf{Notation: } 
We write $\R_{\geq 0}$ for the non-negative real numbers, $I_n$ for the identity in $\R^{n\times n}$, $\spec(M)$ for the complex eigenvalues of a real square matrix $M$, $\|\cdot\|_{\R^n}$ for the standard Euclidean norm, and $\mathcal{C}^r(A,B)$ for the $r$-times continuously differentiable functions from $A$ to $B$. For smooth manifolds $M,N$ and a smooth map $F:M\to N$, the \emph{differential} at $x\in M$ is the linear map $\mathrm{d}F_x:T_xM\to T_{F(x)}N$. If $V:M\to\R$ is smooth and $M$ carries a Riemannian metric $\langle\cdot,\cdot\rangle_x$, the Riemannian gradient $\nabla V(x)\in T_xM$ is defined by $\mathrm{d}V_x[\xi]=\langle \nabla V(x),\xi\rangle_x$ for all $\xi\in T_xM$. For a differentiable map $\psi:\R^{m_n}\times\R^{m_n}\to\R^p$, $D_\omega\psi$ and $D_x\psi$ denote partial Jacobians with respect to $\omega$ and $x$. 

\section{Background and Problem Formulation}
\label{sec:background_problem}
\subsection{Neural Lyapunov Functions on $\R^n$}
\label{sec:backgroundNLF}
\label{sec:neural_lyapunov_function}
Consider a dynamical system
\begin{equation}
    \label{eq:dynamical_system_Rn}
    \dot{x}(t) = f(x(t)), \quad x(0) = x_0 \in \R^n,
\end{equation}
where $f:\R^n \to \R^n$ is Lipschitz continuous and $f(0) = 0$. We assume that there exists a unique forward-complete solution.
A possible neural network architecture for neural Lyapunov functions on $\R^n$ was introduced in \cite{gaby2022lyapunov} as
\begin{equation}
    \label{eq:V_gaby}
    V_{\theta}: \begin{array}[t]{ccl}
        \R^n & \to & \R_{\geq 0} \\
        x & \mapsto & \| \psi_{\theta}(x) - \psi_{\theta}(0) \|_{\R^p} + \alpha \| x \|_{\R^n},
    \end{array}
\end{equation}
where $\psi_{\theta}: \R^n \to \R^p$ is a traditional feedforward neural network, and $\alpha > 0$. 
This neural network architecture is inspired by Lyapunov theory \cite{khalil2002nonlinear}, which requires a Lyapunov function to be positive definite. The first term in~\eqref{eq:V_gaby} $x \mapsto \| \psi_{\theta}(x) - \psi_{\theta}(0) \|_{\R^p}$ is a positive semidefinite function. To ensure that $V_{\theta}$ is positive definite, a term $\alpha \| x \|_{\R^n}$ is added. As proven in \cite{gaby2022lyapunov}, this architecture is a universal approximation of any candidate Lyapunov function on any compact set included in $\R^n$ excluding a neighborhood around the origin.
\subsection{The special orthogonal group $\SO{n}$}
In this paper, we work on the special orthogonal group, defined as follows. For further background on Lie groups, we refer the reader to~\cite{BulLew2004}.

\begin{definition}
The special orthogonal group is defined as
\[
\SO{n} = \left\{ R \in \mathbb{R}^{n \times n} \;\middle|\; R^\top R = I_n, \; \det(R) = 1 \right\}.
\]
Under matrix multiplication, $\SO{n}$ forms a group with identity element $I_n$. 
Moreover, $\SO{n}$ has the structure of a smooth manifold and is therefore a Lie group.
\end{definition}

This manifold is commonly used for the rotation of 3D bodies ($n = 3$), for instance, and the reader can refer to \cite[Chap. 4]{BulLew2004} for some examples. A dynamical system defined on this manifold has the following structure:
\begin{equation}
    \label{eq:dynamical_system1}
    \dot{R}(t) = R(t) \Gamma(t), \quad R(0) = R_0 \in \mathsf{SO}(n),
\end{equation}
where $t \mapsto R(t)$ is the trajectory and $\Gamma = R^{\top} \dot{R}$ is the body velocity. This velocity belongs to the Lie algebra $\so{n}$, which is defined next. 

\begin{definition}
    The Lie algebra $\so{n}$ is defined as the tangent space at the identity element $I_n$ of $\SO{n}$, i.e.,
    \(
        \so{n} = \left\{ S \in \R^{n \times n} \;\middle|\; S^\top = -S \right\}.
    \)
    This set forms the real vector space of \emph{skew-symmetric} matrices and has dimension $m_n = \frac{n(n-1)}{2}$.
\end{definition}
The Lie algebra $\so{n}$ is isomorphic to $\R^{m_n}$ \cite[p.~252]{BulLew2004} and the isomorphism is realized by the \emph{hat} and \emph{vee} maps:
\begin{equation}
    \label{eq:hat_vee}
    \widehat{(\cdot)} : \R^{m_n} \rightarrow \so{n}, \quad \quad (\cdot)^\vee : \so{n} \rightarrow \R^{m_n}.
\end{equation}
Since $\SO{n}$ is a Lie group, it admits a left-invariant Riemannian metric~\cite{BulLew2004}. Under this metric, the inner product on each tangent space $T_R\SO{n}$ at $R$, denoted $\langle \cdot, \cdot \rangle_R$, is related to the inner product at identity $\langle \cdot, \cdot \rangle_{I_n}$. In particular, for $\xi_R, \eta_R \in T_R \SO{n}$ with $\xi_R = R\widehat\Omega_1$ and $\eta_R = R\widehat\Omega_2$, where $\widehat\Omega_1,\widehat\Omega_2 \in \so{n}$, the following property holds:
\begin{align}
\langle \xi_R, \eta_R \rangle_R
&= \langle R\widehat\Omega_1, R\widehat\Omega_2 \rangle_R \label{eq:bi-invariant1}\\
& =\langle R^\top  R\widehat\Omega_1, R^\top R\widehat\Omega_2 \rangle_{I_n} 
=
\langle \widehat\Omega_1, \widehat\Omega_2 \rangle_{I_n} \label{eq:bi-invariant2} \\
&= \langle {\Omega}_1, {\Omega}_2 \rangle_{\mathbb{R}^{m_n}} \label{eq:bi-invariant3}.
\end{align}
The inner products in~\eqref{eq:bi-invariant1},~\eqref{eq:bi-invariant2}, and~\eqref{eq:bi-invariant3}, are on $T_R\SO{n}$, $T_{I_n}\SO{n}$ (isomorphic to $\so{n}$), and $\R^{m_n}$,~respectively.

\subsection{Dynamically controlled systems on $\SO{n}$}
We consider a dynamically controlled system with configuration on $\SO{n}$ using the
left-trivialized representation of its tangent bundle. Since
$\so{n}=T_{I_n}\SO{n}$ is identified with $\R^{m_n}$ by the isomorphism
in~\eqref{eq:hat_vee}, each $x\in\R^{m_n}$ defines
$\widehat{x}\in\so{n}$ and hence the tangent vector $R\widehat{x}\in T_R\SO{n}$.
Thus $R\in\SO{n}$ is the configuration, $x$ is the body velocity, and,
under left trivialization, the tangent-bundle state is identified with
$(R,x)\in\SO{n}\times\R^{m_n}$. We consider
\begin{equation}
\label{eq:dynamical_system2}
    \left\{
        \begin{array}{ll}
            \dot{R}(t)=R(t)\widehat{x}(t), 
            & R(0)=R_0\in\SO{n},\\
            \dot{x}(t)=f(x(t),R(t)), 
            & x(0)=x_0\in\X,
        \end{array}
    \right.
\end{equation}
where $\X\subset\R^{m_n}$ is a compact connected domain containing the origin, and
$f:\SO{n}\times \X\to\R^{m_n}$. Here, $\widehat{x}(t)$ denotes the pointwise application of the hat map to $x(t)$.
We assume that there exists a unique forward-complete solution $\pi(\cdot, R_0, x_0)$ for any initial condition $(R_0, x_0) \in \K \triangleq \SO{n} \times \X$ \cite{BulLew2004}.

Without loss of generality, we assume that the dynamical system~\eqref{eq:dynamical_system2} has $e = \left(I_n,  0 \right)$ as an equilibrium point. The problem addressed in this paper is to characterize the maximal region of attraction around $e$ in $\K$ defined as:
\begin{equation*}
    \mathcal{R}^* = \left\{ (R_0, x_0)\in \K \ \Big| \begin{array}{l} \lim_{t \to \infty} \pi(t, R_0, x_0) = e, \\ \forall t \geq 0, \ \pi(t, R_0, x_0) \in \K \end{array} \right\}.
\end{equation*}

\section{Neural Lyapunov Functions on $\SO{n}$}
\label{sec:neural_lyapunov_son}
In this section, we first derive the neural architecture to model a positive-definite function on $\K$, then we discuss the conditions under which it is a local Lyapunov function for the dynamical system in \eqref{eq:dynamical_system2} around $e$. Finally, we present our learning problem.

\subsection{Neural Lyapunov function candidate}

The Lyapunov direct method for assessing asymptotic stability of the equilibrium point $e$ in $\K$ requires the existence of a candidate Lyapunov function defined as a positive definite function on $\K$ with $e$ as the origin.

Inspired by \eqref{eq:V_gaby}, let a candidate Lyapunov function on the compact set $\K$ be written as
\begin{equation}
    \label{eq:candidate_lyap}
    V: \begin{array}[t]{ccl}
        \K & \!\!\to\!\! & \R_{\geq 0} \\
        (R,x) & \!\!\mapsto\!\! & \left\| \Psi(R, x) - \Psi(e) \right\|_{\R^p}^2 + \ \alpha \| (R, x)\|_\K^2,
    \end{array}
\end{equation}
where $\Psi$ is continuously differentiable everywhere but at the origin and locally Lipschitz, $\alpha > 0$, and $\| (R, x) \|_\K^2 = \|R - I_n\|_F^2 + \|x\|_{\R^{m_n}}^2$. 
Compared with \eqref{eq:V_gaby}, this formulation considers a heterogeneous state on $\K$ and that the norms are now squared. This form has better approximation properties since $V$ is continuously differentiable.

On $\R^n$, $\Psi$ can be chosen as a feedforward neural network and can therefore approximate arbitrarily closely any locally Lipschitz function and its gradient on a compact set of $\R^n$ \cite{hornik1991approximation}. However, there is no such approximation result on $\SO{n}$ since it is not globally isomorphic to a Euclidean space.

However, we have seen in the previous section that $\so{n}$ is isomorphic to $\R^{m_n}$. We therefore  define the following set:
\begin{equation}\label{eq:SO_reg}
    \SO{n}_{\mathrm{reg}} \eqdef \left\{ R \in \SO{n} \ | \ \exists \; \omega \in \mathcal{D}_n, \ R = \Exp(\widehat{\omega}) \right\},
\end{equation}
where $\Exp: \so{n} \to \SO{n}$ is the matrix exponential and
\begin{equation}
\label{eq:D-n}
    \mathcal{D}_n = \left\{ \omega \in \R^{m_n} \ | \ \| \widehat{\omega} \|_2 < \pi \right\}.
\end{equation}
With this definition, we have $\SO{n}_{\mathrm{reg}} = \{ R \in \SO{n} \ | \ -1 \not\in \spec(R) \}$ which excludes rotations by exactly $\pi$, for which the logarithm is not unique~\cite{GallXu2003}. Similarly, we define \(
    \so{n}_{\mathrm{reg}} \triangleq
    \left\{ \widehat{\omega}\in\so{n} \mid \omega\in\mathcal{D}_n \right\}
\). Therefore, we can define the matrix logarithm $\Log$ as the inverse of the exponential on $\SO{n}_{\mathrm{reg}}$:
\[
    \forall R \in \SO{n}_{\mathrm{reg}}, \quad R = \Exp\left(\Log(R)\right).
\]

Consequently, the proposed logarithmic-map construction is well-defined on $\SO{n}_{\mathrm{reg}}$.
We can now use the map \textit{vee}, given in~\eqref{eq:hat_vee}, to construct a bijection between $\SO{n}_{\mathrm{reg}}$ and $\mathcal{D}_n$ as:
\begin{equation}
\label{eq:gamma_map}
    \forall R \in \SO{n}_{\mathrm{reg}}, \exists \; \omega \in \mathcal{D}_n, \quad \omega = \Log(R)^\vee = \gamma(R).
\end{equation}
The mappings $\Exp,\gamma$, $\widehat{(\cdot)}$, and  their inverse maps are summarized in Figure~\ref{fig:gamma}.
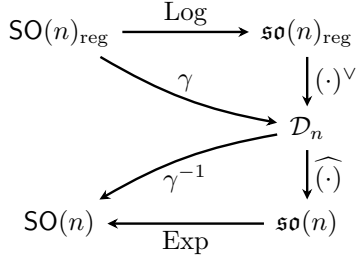
\begin{figure}
    \centering
\[
\begin{tikzcd}[
column sep=4.8em,
row sep=1.9em,
cells={nodes={scale=1.0}},
every arrow/.append style={line width=0.9pt},
every label/.append style={font=\normalsize}
]
\SO{n}_{\mathrm{reg}}
  \arrow[r, "\Log"]
  \arrow[dr, bend right=10, "\gamma"{above}]
&
\so{n}_{\mathrm{reg}}
  \arrow[d, "(\cdot)^\vee"]
\\
{} &
\mathcal{D}_n
  \arrow[dl, bend right=10, "\gamma^{-1}"{below}]
  \arrow[d, "\widehat{(\cdot)}"]
\\
\SO{n} &
\so{n}
  \arrow[l, "\Exp"]
\end{tikzcd}
\]
    \caption{Mapping between the different spaces. The top transformations are used for inference of $V_{\theta_1,\alpha}$, while the lower mappings are used for sampling in $\SO{n}$.}
    \label{fig:gamma}
    \vspace*{-0.5cm}
\end{figure}
We also define
\(
    \K_{\mathrm{reg}} \triangleq \SO{n}_{\mathrm{reg}}\times\X .
\)
This construction certifies $\mathcal R^*\cap\K_{\mathrm{reg}}$.
These maps enable the formulation of the following proposition on the approximation of the candidate Lyapunov function on $\K_{\mathrm{reg}}$. 

\begin{proposition}
\label{prop:Optimal_lyapunov}
    If $V^* \in \mathcal{C}^{2}(\K_{\mathrm{reg}}, \R)$ is a positive-definite function lower-bounded by $\beta \|\cdot\|_\K^2$ on $\K_{\mathrm{reg}}$ for $\beta > 0$ with $V^*(e) = 0$, then there exist $\Psi$ locally Lipschitz on $\K_{\mathrm{reg}}$ and $\alpha > 0$ such that $V^*$ can be written as \eqref{eq:candidate_lyap}. 
\end{proposition}
\begin{proof}
    For any $0 < \alpha < \beta$, define $\tilde{V}(R,x) = V^*(R,x) - 
    \alpha \| (R,x) \|_{\K}^2 \geq (\beta - \alpha) \|(R,x) \|_{\K}^2 \geq 0$.
    Set $\psi_1(R,x) = \sqrt{\tilde{V}(R,x)}$ and $\psi_i \equiv 0$ for $i \in \{2, \dots, p\}$.
    Then for $\Psi = [\psi_1 \cdots \psi_p]^{\top}$, we get $V^* = V$ for $V$ as in 
    \eqref{eq:candidate_lyap}. It remains to show that $\psi_1 = \sqrt{\tilde{V}}$ is locally Lipschitz on $\K_{\mathrm{reg}}$. 
    Since $\K_{\mathrm{reg}}$ is a smooth manifold and every point admits a
coordinate neighborhood, it suffices to show Lipschitz continuity in local
coordinates around every point.
    Let $q \in \K_{\mathrm{reg}}$ and let $\Phi (R,x) = [\gamma(R)^{\top} \ x^{\top}]^{\top}$, which is a $C^1$-diffeomorphism from a neighborhood $\mathcal{N}$ of $q$ in $\K_{\mathrm{reg}}$ to an open subset of $\R^{m_n}\times\R^{m_n}$. Let $g = \tilde{V} \circ \Phi^{-1}$, which is $C^1$ on $\Phi(\mathcal{N})$.

    \textit{Away from $e$:} If $q \neq e$, let $\mathcal{N}$ be a neighborhood of $q$ not containing $e$. Then $g > 0$ on $\Phi(\mathcal{N})$, so $\sqrt{g}$ 
    is $C^1$ with Euclidean gradient $\nabla \sqrt{g} = \frac{\nabla g}{2\sqrt{g}}$, which is bounded on any compact subset of $\Phi(\mathcal{N})$. 

    \textit{Near $e$:} If $q = e$, then $g(0) = 0$ and, since $V^*(e) = 0$ is a minimum 
    of $V^* \in C^2$, we have $\nabla g(0) = 0$, so $\|\nabla g(v)\| = O(\|v\|)$ near 
    zero. The lower bound on $\tilde{V}$ gives $g(v) \geq c\|v\|^2$ for some $c > 0$, 
    hence $\sqrt{g(v)} \geq \sqrt{c}\|v\|$. Therefore,
    \begin{equation*}
        \|\nabla \sqrt{g}(v)\| = \frac{\|\nabla g(v)\|}{2\sqrt{g(v)}} 
        \leq \frac{O(\|v\|)}{2\sqrt{c}\,\|v\|} = O(1),
    \end{equation*}
    so $\sqrt{g}$ is locally Lipschitz at $0$ in $\R^{m_n}\times\R^{m_n}$.
\end{proof}

\begin{proposition}
\label{prop:Lyapun_P_log}
Let $\psi_\theta:\R^{m_n}\times\R^{m_n}\to\R^p$ be a neural network, and for $\alpha > 0$
\begin{equation}
\label{eq:P_lyapu}    
\mathcal{P}_{\theta,\alpha}(R,x)
=
\|\psi_\theta(\gamma(R), x) - \psi_\theta(0,0)\|^2
+
\alpha\|(R,x)\|_\K^2.
\end{equation}
If \eqref{eq:dynamical_system2} admits a local quadratic Lyapunov function, then any Lyapunov function $V^* \in \mathcal{C}^2(\K_{\mathrm{reg}}, \R)$ (and its gradient) can be approximated arbitrarily closely in the $L_{\infty}$ norm by $\mathcal{P}_{\theta,\alpha}$ (and its gradient) on any compact set $K \subset \K_{\mathrm{reg}}$.%
\end{proposition}

\begin{proof}
The positive-definiteness of $\mathcal{P}_{\theta,\alpha}$ follows directly from
\eqref{eq:P_lyapu} since both terms are nonnegative and the second one is positive definite.
We now clarify the approximation statement. Let
$K\subset\K_{\mathrm{reg}}$ be compact. 
Since $(R,x)\mapsto(\gamma(R),x)$ is smooth on $\K_{\mathrm{reg}}$, the image
of any compact set $K\subset\K_{\mathrm{reg}}$ is compact in
$\mathbb{R}^{2m_n}$.
Let $V_q$ be a local quadratic Lyapunov function around $e$. Then $V(R,x) = V^*(R,x) + \kappa V_q$ approximates $V^*$ arbitrarily closely in the $L_\infty$ norm on any compact set $K$ as $\kappa > 0$ is chosen sufficiently small, and is lower-bounded by a positive-definite quadratic function. 
By Proposition~\ref{prop:Optimal_lyapunov}, there exists $\Psi$ locally Lipschitz on $\K_{\mathrm{reg}}$ and continuously differentiable on $\K_{\mathrm{reg}}\setminus\{e\}$ such that \eqref{eq:candidate_lyap} holds.
This $\Psi$ and its gradient can be approximated arbitrarily closely in the $L_{\infty}$ norm on any compact set $K \subset \K_{\mathrm{reg}}$ by a feedforward neural network~\cite{hornik1991approximation}. 
This leads to the universal approximation of $V^*$.
\end{proof}
Assuming $f$ is continuously differentiable is sufficient for the existence of a $C^2$ Lyapunov function \cite{khalil2002nonlinear,Sideris2013}, and Proposition~\ref{prop:Lyapun_P_log} gives us a candidate Lyapunov function $V_{\theta_1,\alpha} = \mathcal{P}_{\theta_1,\alpha}$ on $\K_{\mathrm{reg}}$ with universal approximation capabilities. We will investigate how to transform it into a Lyapunov function. 

\subsection{Zubov characterization}

To establish the local asymptotic stability of the equilibrium point $e$, the Lyapunov direct method (such as~\cite[Theorem~6.45]{BulLew2004}) requires us to show that the Fréchet derivative of $t \mapsto V_{\theta_1,\alpha}(R(t),x(t))$ is negative definite along the trajectories of \eqref{eq:dynamical_system2}.
In a local neighborhood of $e$, this derivative must be strictly negative within the $1$-level set of $V_{\theta_1,\alpha}$, i.e.,
\begin{multline}
\label{eq:V-theta-alpha}
    \forall v_0 \in \K_{\mathrm{reg}} \setminus \{e\}, \\
    V_{\theta_1,\alpha}(v_0) < 1 \Rightarrow \forall t > 0, \ \frac{d}{dt} V_{\theta_1,\alpha}(\pi(t, v_0)) < 0.
\end{multline}

The Lyapunov direct method  enables us to conclude that the $1$-level set of $V_{\theta_1,\alpha}$ defined as 
\begin{equation}
\label{eq:ROA-set}
    \mathcal{R}(V_{\theta_1,\alpha}) = \left\{ (R,x) \in \K_{\mathrm{reg}} \ | \ V_{\theta_1,\alpha}(R, x) < 1 \right\}
\end{equation}
is a region of attraction around the equilibrium point $e$. 

In this paper, we are interested in estimating the maximal region of attraction $\mathcal{R}^*$. Zubov's theorem~\cite[Theorem~22]{zubov1961methods} provides a practical characterization of the maximal region of attraction as the $1$-level set of a maximal Lyapunov function $V^*$ defined as a solution to
\begin{multline}
    \label{eq:frechet_zubov}
    \forall v_0 \in \K_{\mathrm{reg}} \setminus \{e\}, V^*(v_0) < 1 \Rightarrow \ \forall t > 0, \\
     \ \frac{d}{dt} V^*(\pi(t, v_0)) = - \phi(\pi(t,v_0))\left(1 - V^*(\pi(t,v_0)) \right),
\end{multline}
where $\phi$ is a  positive definite function on $\K$. This formulation is still dependent on time and not just on the state. We need a way to compute the Fréchet derivative $t \mapsto \frac{d}{dt} V^*(\pi(t,v))$, defined at any point $v = (R,x) \in \K$. Let $(R_t, x_t) = \pi(t, R, x)$, and by definition of the gradient and using the chain rule, we get:
\begin{equation}
    \label{eq:dVdt}
    \hspace*{-0.3cm}
    \frac{d}{dt} V^*(R_t,x_t) \!\!\!
    \begin{array}[t]{l} 
       \  = \langle \nabla_R V^*(R_t,x_t), \dot{R}_t \rangle_{R_t} \\
       \hfill + \langle \nabla_x V^*(R_t,x_t), \dot{x}_t \rangle_{\R^{m_n}} \\
       \overset{\mathrm{by\,}\eqref{eq:dynamical_system2}}{=} 
       \langle \nabla_R V^*(R_t,x_t), R_t \widehat{x_t} \rangle_{R_t} \\
       \hfill + \ \langle \nabla_x V^*(R_t,x_t), f(x_t,R_t) \rangle_{\R^{m_n}} \\
       \overset{\mathrm{by\,}\eqref{eq:bi-invariant2}}{=} 
       \langle R_t^\top \nabla_R V^*(R_t,x_t), R_t^\top R_t\widehat{x_t} \rangle_{I_{n}} \\
       \hfill + \ \langle \nabla_x V^*(R_t,x_t), f(x_t,R_t) \rangle_{\R^{m_n}} \\
       \overset{\mathrm{by\,}\eqref{eq:bi-invariant3}}{=} 
       \langle \left(R_t^\top \nabla_R V^*(R_t,x_t)\right)^{\vee}, x_t \rangle_{\R^{m_n}} \\
       \hfill + \ \langle \nabla_x V^*(R_t,x_t), f(x_t,R_t) \rangle_{\R^{m_n}}.
    \end{array}
\end{equation}
The above expression allows us to compute the Fr\'echet derivative using the gradient information, and to transform the previous problem~\eqref{eq:frechet_zubov} into a time-independent equivalent problem, as stated below.
\begin{proposition}
    If there exists a positive definite function $\phi$ on $\K$ such that $V^*$ is a positive definite solution to
    \begin{multline}
        \label{eq:zubov}
        \langle \left(R^\top \nabla_R V^*(R,x)\right)^{\vee}, x \rangle_{\R^{m_n}} \\
        + \langle \nabla_x V^*(R,x), f(x,R) \rangle_{\R^{m_n}} = \\ 
        -\phi(R, x)\big(1-V^*(R, x)\big),
    \end{multline}
    then $\mathcal{R}^* = \mathcal{R}(V^*)$.
\end{proposition}
The complete proof of this proposition is an adaptation of \cite[Corollary~3.1]{vannelli1985maximal} on $\K\subset\SO{n}\times\R^{m_n}$, using the time derivative of the Lyapunov function in \eqref{eq:dVdt}.

This formulation characterizes the maximal Lyapunov function $V^*$.
In practice, we aim to approximate $V^*$ by the candidate Lyapunov
function $V_{\theta_1,\alpha}$ in~\eqref{eq:V-theta-alpha} through a learning procedure as described below.

\subsection{Learning formulation}

To estimate the maximal region of attraction $\mathcal{R}^*$, we will first approximate a maximal Lyapunov function $V^*$ by $V_{\theta_1,\alpha} = \mathcal{P}_{\theta_1,\alpha}$ and $\phi$ by $\phi_{\theta_2} = \mathcal{P}_{\theta_2,\beta}$, where $\theta_1, \theta_2$, and $\alpha$ are training variables, while $\beta$ is fixed to a small positive value. If $V_{\theta_1,\alpha}$ is close to a $V^*$ and is a Lyapunov function for \eqref{eq:dynamical_system2}, then $\mathcal{R}(V_{\theta_1,\alpha})$ in~\eqref{eq:ROA-set} is an inner approximation of the maximal region of attraction $\mathcal{R}^*$.

Learning a maximal Lyapunov function on a Euclidean space has been investigated in \cite{liu2023physics,barreau2024supervised}. We will follow the same procedure and define a learning problem that uses the neural architecture proposed previously while enforcing \eqref{eq:zubov} on $\K_{\mathrm{reg}}$. 
We define a quantity, called the residual, which depends on the parameters $\theta_1, \theta_2, \alpha$, and~$\beta$ as follows:
\begin{multline*}
    r(R,x) \triangleq
    \left\langle
    \left(R^\top\nabla_R V_{\theta_1,\alpha}(R,x)\right)^\vee,
    x
    \right\rangle_{\R^{m_n}} \\
    +
    \left\langle
    \nabla_x V_{\theta_1,\alpha}(R,x),
    f(x,R)
    \right\rangle_{\R^{m_n}} \\
    +
    \phi_{\theta_2,\beta}(R,x)
    \big(1-V_{\theta_1,\alpha}(R,x)\big).
\end{multline*}
To estimate a maximal Lyapunov function $V^*$, the constraint \eqref{eq:zubov} will be softly enforced through $r = 0$ on a discrete subset of $\K_{\mathrm{reg}}$. While the weaker condition $r\leq 0$ would suffice to certify a Lyapunov function on a sublevel set, the equality $r=0$ is used here because Zubov's characterization of the maximal Lyapunov function is an equality. To that extent, we define $\mathcal{D}(N) \subset \mathcal{D}_n \times \X$ as $N$ uniformly sampled points from $\mathcal{D}_n \times \X$ in~\eqref{eq:D-n}, and therefore 
\begin{equation}
    \label{eq:D_N}
    \tilde{\mathcal{D}}(N) = \left\{ (\gamma^{-1}(\omega), x) \ | \ (\omega,x) \in \mathcal{D}(N) \right\}.
\end{equation}
Using this definition, the learning problem is written as
\begin{equation}
    \label{eq:learning}
    (\theta_1^*(\tilde{\mathcal{D}}(N)), \alpha^*) \in \argmin_{\theta_1,\alpha} \min_{\theta_2} \frac{1}{N} \sum_{(R, x) \in \tilde{\mathcal{D}}(N)} \| r(R, x) \|^2.
\end{equation}
If the residual $r$ converges to zero
for sufficiently dense samples of $\K_{\mathrm{reg}}$,
the learned function $V_{\theta_1,\alpha}$ approximates a maximal
Lyapunov function and the associated level set
$\mathcal{R}(V_{\theta_1,\alpha})$ provides an estimate of the maximal region of attraction. It is, however, difficult to compute the loss because the gradient of $V_{\theta_1,\alpha}$ has to be made explicit first.

In classical physics-informed learning scenarios, this gradient is computed through automatic differentiation \cite{baydin2018automatic}. However, in our case, it requires evaluating the gradient of the logarithmic map. Since this operation is not directly implemented in classical automatic differentiation libraries, we derive explicit formulas for the derivative of the logarithmic map~\eqref{eq:gamma_map} in the next section.

\section{Derivative of the Logarithmic Map}
\label{sec:log_derivative}

For $(R, x) \in \K_{\mathrm{reg}}$, let $(R(t), x(t)) = \pi(t, R, x)$ and $X(t) = \Log(R(t))$, which implies $R(t) = \Exp(X(t))$. The automatic differentiation framework requires the quantity (explained below in~\eqref{eq:G}) that gives the changes in $X$ due to a small variation in $R$. Consider the curve
\(
R_\delta \triangleq R\Exp(\delta dR)\), where \(dR \in \so{n}\). Then
\begin{equation}
\label{eq:R_del_dot}
\dot R_\delta\big|_{\delta=0}
=
R\frac{d}{d\delta}\Exp(\delta dR)\big|_{\delta=0}
=
R\,dR
\ \in T_R\SO{n}.
\end{equation}
Let \(X(\delta)\triangleq\Log(R_\delta)\). By definition of the directional derivative, we can write the left-trivialized variation
\begin{equation} 
    \label{eq:G}
    G^\ell(R,dR) \begin{array}[t]{l}
        \displaystyle\triangleq \lim_{\delta\to 0} \frac{\Log(R\Exp(\delta dR))-\Log(R)}{\delta} \\ 
        = \left.\frac{d}{d\delta}\right|_{\delta=0} X(\delta) = R^{\top} G(R,dR),
    \end{array}
\end{equation}
where the superscript $\ell$ denotes left trivialization.
The quantity $G$ is the one needed in TensorFlow to allow for the computation of $\nabla V_{\theta_1,\alpha}$.

\subsection{Derivative of the logarithmic map}
The derivative of the $\Log$ map along trajectories of
\eqref{eq:dynamical_system2} is given by the following result.

\begin{proposition}
\label{prop:log_derivative}
Let $t \mapsto (R(t), x(t))$ be a trajectory of the system defined in \eqref{eq:dynamical_system2} and $X(t)=\Log(R(t))$.
Then
\begin{equation}
\dot X(t) = (\dd^\ell\Exp_{X(t)})^{-1}[\widehat{x}(t)],
\label{eq:Xdot}
\end{equation}
where the left-trivialized differential of the exponential map $\dd^\ell\Exp_X : \so{n} \to \so{n}$ is
\[
\dd^\ell\Exp_X[\Delta]
=
\int_0^1 e^{-sX}\Delta e^{sX}ds, \qquad
\Delta \in \so{n}.
\vspace*{-0.85cm}
\]
\end{proposition}


\begin{proof}
Differentiating the identity $R(t)=\Exp(X(t))$ yields
\[
\dot R(t)=\dd\Exp_{X(t)}[\dot X(t)].
\]
From \cite[Prop.~20.1]{GalQua2020}, we get that $\dd^\ell\Exp_{X(t)}$ is non-singular and we can use the identity
\(
\dd\Exp_X[\Delta]=\Exp(X)\dd^\ell\Exp_X[\Delta].
\)
Together with the system dynamics \eqref{eq:dynamical_system2}, it gives
\[
\Exp(X(t))\dd^\ell\Exp_{X(t)}[\dot X(t)]
=
\Exp(X(t))\widehat{x}(t).
\]
Premultiplying by $\Exp(-X(t))$ yields the result.
\end{proof}
Since \(R_\delta=\Exp(X(\delta))\), it follows from Proposition~\ref{prop:log_derivative} that
\(
\dot X(\delta)
=
(\dd^\ell\Exp_{X(\delta)})^{-1}[R_\delta^{-1}\dot R_\delta].
\)
Evaluating at \(\delta=0\), and using \eqref{eq:R_del_dot}-\eqref{eq:G} yields
\[
G(R,dR)
=
R (\dd^\ell\Exp_{\Log(R)})^{-1}[dR]
\in T_R\SO{n}.
\]
Using \eqref{eq:Xdot}, the left-trivialized $R$-gradient appearing in the
residual can be written explicitly. Let $\omega=\gamma(R)$,
$X=\Log(R)$, and
$z_{\theta_1}(\omega,x)=\psi_{\theta_1}(\omega,x)-\psi_{\theta_1}(0,0)$.
Define
\[
J_\gamma^\ell(R)\eta
=
\left((\dd^\ell\Exp_X)^{-1}[\widehat\eta]\right)^\vee,
\qquad \eta\in\mathbb{R}^{m_n}.
\]
Then
\begin{equation*}
\begin{aligned}
    \left(R^\top\nabla_R V_{\theta_1,\alpha}(R,x)\right)^\vee
    &=
    \left(J_\gamma^\ell(R)\right)^\top
    \nabla_\omega
    \Big(
        \|z_{\theta_1}(\omega,x)\|^2 \\
        &\qquad\qquad
        +\alpha\|\Exp(\widehat\omega)-I_n\|_F^2
    \Big),
\end{aligned}
\end{equation*}
where
$\nabla_\omega\|z_{\theta_1}(\omega,x)\|^2
=
2D_\omega\psi_{\theta_1}(\omega,x)^\top z_{\theta_1}(\omega,x)$.
Similarly,
$\nabla_x V_{\theta_1,\alpha}(R,x)
=
2D_x\psi_{\theta_1}(\omega,x)^\top z_{\theta_1}(\omega,x)+2\alpha x$.

\subsection{Numerical evaluation}

Computing $(\dd^\ell\Exp_{X(t)})^{-1}[dR]$ naively requires discretizing an integral and inverting it. This is inefficient and too inaccurate for a learning problem. However, there are other strategies that yield approximations of any order and exact computation, as described below~\cite{GalQua2020}.

\textbf{Series approximation.}
Using the adjoint representation $\mathrm{ad}_X[\Delta]=[X,\Delta]$, where $[\cdot,\cdot]$ denotes the Lie bracket,  the inverse operator admits the Bernoulli series expansion
\begin{multline*}
    (\dd^\ell\Exp_X)^{-1}[dR]
=
dR
+\frac12[X,dR]\\
+
\sum_{m=1}^{\infty}
\frac{B_{2m}}{(2m)!}(\mathrm{ad}_X)^{2m}[dR],
\end{multline*}
where $B_{2m}$ are the even Bernoulli numbers. Since $R\in\SO{n}_{\mathrm{reg}}$, the spectrum of $\mathrm{ad}_X$ excludes $\pm i2k\pi$, ensuring absolute convergence of the series. In practice, truncating the series provides a fast approximation.
Each $\mathrm{ad}_X$ evaluation requires two $n\times n$ matrix
multiplications, i.e., $\mathcal{O}(n^3)$ operations.
Truncating the series after $M$ terms, therefore, yields complexity
$\mathcal{O}(Mn^3)$, which is $\mathcal{O}(n^3)$ in practice since only a few terms are typically required.

\begin{remark}
Note that, due to numerical roundoff, the truncated approximation of 
    $(\dd^\ell\Exp_X)^{-1}[dR]$ may fail to be exactly skew-symmetric. 
    In the code, we enforce skew-symmetry in the approximation.
\end{remark}

\textbf{Exact computation.}
An exact evaluation can be obtained using the block matrix identity
\[
\Exp\left(
\begin{bmatrix}
X & \Delta \\
0 & X
\end{bmatrix}
\right)
=
\begin{bmatrix}
e^X & e^X\dd^\ell\Exp_X[\Delta]\\
0 & e^X
\end{bmatrix}.
\]
This relation allows the construction of the linear operator
$\dd^\ell\Exp_X$ in a chosen basis of $\so{n}$ and the exact solution of
$\dd^\ell\Exp_X[\Delta]=Y$.
Specifically, to compute $(\dd^\ell\Exp_X)^{-1}[Y]$, let
$\{E_i\}_{i=1}^{m_n}$ be a basis of $\so{n}$. For each $j$, compute
\[
P_j=\Exp\!\left(
\begin{bmatrix}
X & E_j\\
0 & X
\end{bmatrix}
\right),
\]
extract the top-right block $B_j$, and set
$V_j=e^{-X}B_j=\dd^\ell\Exp_X[E_j]$.
Express $V_j=\sum_{i=1}^{m_n}J_{ij}E_i$ to form the Jacobian
$J\in\R^{m_n\times m_n}$.
Let $y$ denote the coordinate vector of
$Y=\sum_i y_iE_i$.
Solving $J\delta=y$ yields
\[
(\dd^\ell\Exp_X)^{-1}[Y]
=
\sum_{i=1}^{m_n}\delta_iE_i .
\]
The dominant cost is the computation of $m_n$ block matrix exponentials of size $2n\times2n$.
Using scaling--squaring~\cite{HigNicm2008}, each exponential requires
$\mathcal{O}(n^3)$ operations, giving overall complexity
$\mathcal{O}(m_n n^3)=\mathcal{O}(n^5)$. The system $J\delta=y$ can be solved in
$\mathcal{O}(m_n^3)$ time, which is negligible for the small dimensions typical of
$\so{n}$. The network input is intrinsic, $(\gamma(R),x)\in\mathbb R^{2m_n}$,
rather than an ambient representation of dimension $n^2+m_n$.
\subsection{Final learning algorithm}

Since the loss is now made numerically computable, a solution to \eqref{eq:learning} can be approximated via a gradient-descent scheme with frequent resampling to avoid a sub-optimal $\theta_1^*$ \cite{barreau2025control}.
In the proposed learning algorithm, the series approximation is used during early training iterations to reduce computational cost, while the exact method is employed during the final phase.

\section{Simulation Results}
\label{sec:simulation_results}

We use the proposed methodology to learn an approximation of the maximal region of attraction of the dynamical system introduced in~\eqref{eq:dynamical_system2} with $n = 2$, $m_n = 1$, and $f(x, R) = -x + x^2 - 2(R - R^{\top})^\vee$. If $x = - 2(R - R^{\top})^\vee$, the dynamics on $R$ have $I_2$ as an almost-globally asymptotically stable equilibrium point on $\SO{2}_{\mathrm{reg}}$ \cite{AkhSan2022}, so the coupled system is chosen to have a locally asymptotically stable origin $e$. The code was implemented in Python with the TensorFlow library \cite{abadi2016tensorflow}, and it is available on GitHub\footnote{\texttt{https://github.com/mBarreau/LyapunovSO}}.
\begin{figure}
     \centering
     \begin{subfigure}[b]{0.4\textwidth}
         \centering
         \includegraphics[width=\textwidth]{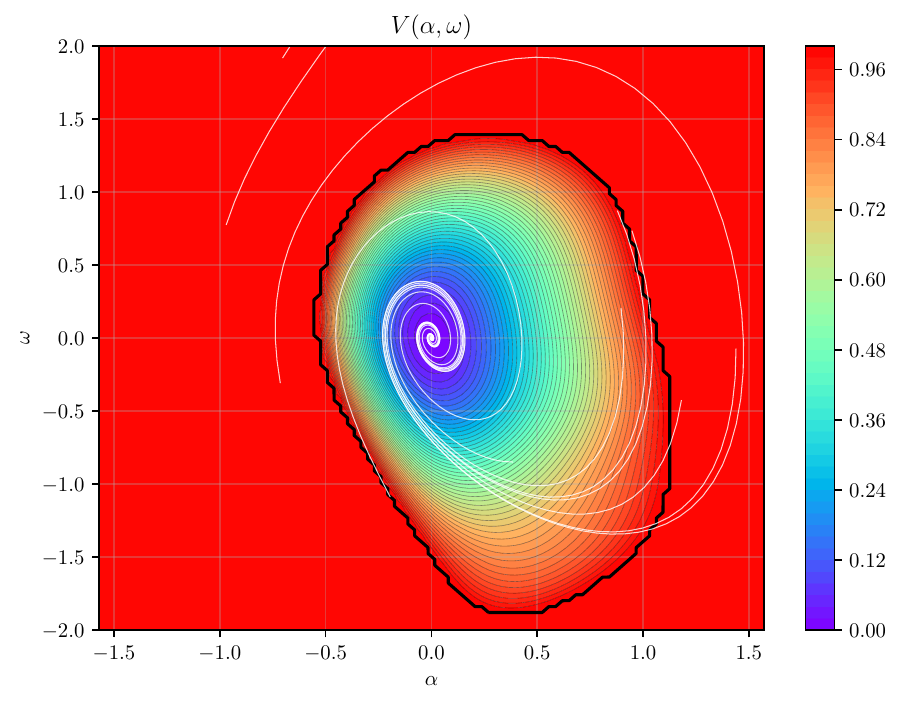}
         \caption{$V_{\theta_1^*,\alpha^*}$}
         \label{fig:V_SO2}
     \end{subfigure}
     \vspace{0.5em}
\begin{subfigure}[b]{0.4\textwidth}
         \centering \includegraphics[width=\textwidth]{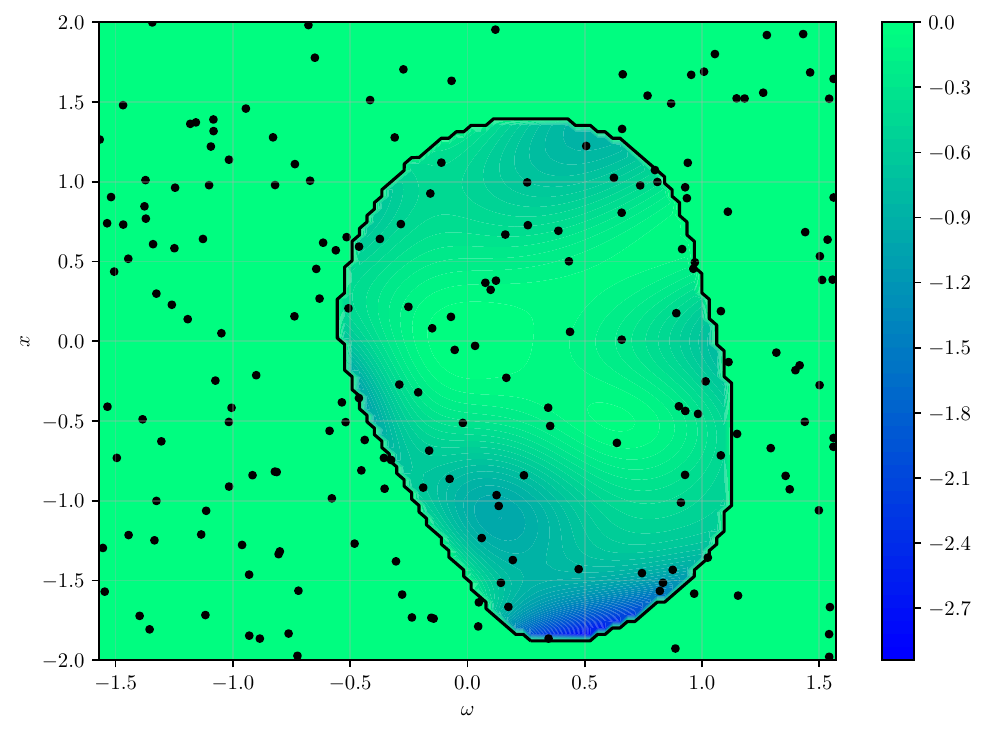}
         \caption{$\dot{V}_{\theta_1^*,\alpha^*}$}
         \label{fig:Vdot_SO2}
     \end{subfigure}
\caption{Learned Lyapunov function and its Fr\'echet derivative.} 
    \label{fig:SO2}
    \vspace*{-0.5cm}
\end{figure}
With $\phi$ fixed as a quadratic function (to simplify the computations), we get the Lyapunov function in Figure~\ref{fig:V_SO2} where $\omega = \gamma(R) \in [-\pi/2, \pi/2]$ and $x \in \X = [-2, 2]$. Some trajectories are plotted in white. We observe that the sampled Fr\'echet derivative of the Lyapunov function is negative definite as expected, and the region of attraction is (inner-)estimated. The percentage of coverage of $\mathcal{R}(V_{\theta_1^*,\alpha^*})$ compared to $\mathcal{R}^*$ is $64.3\%$, which showcases the feasibility of the approach. Considering only the quadratic term in \eqref{eq:P_lyapu} ($\psi_\theta \equiv 0$) leads to an infeasible learning problem for that example, underscoring the method's interest.

\section{Conclusion and Future Work}
\label{sec:conclusion}
We proposed a learning framework for designing neural Lyapunov functions on $\SO{n}$. 
We formulated a class of neural Lyapunov functions on $\SO{n}$ and theoretically derived its approximation capabilities. A central technical contribution was the derivation of explicit, numerically tractable expressions for the derivative of the logarithmic map, enabling an efficient two-phase training procedure that balances computational cost and accuracy.
The results demonstrate the approach on an interpretable low-dimensional
tangent-bundle example. Future work will address computational speed-ups and
systematic validation on larger Lie-group systems.

\bibliographystyle{IEEEtran}
\bibliography{bib_adeel,bib_matth}

\end{document}